\documentclass[12pt,reqno]{amsart}
\usepackage{amsmath}
\usepackage{amsthm}
\usepackage{amssymb}
\usepackage[all]{xy}
\usepackage{color}
\usepackage{tikz}
\usepackage{mathrsfs}
\usepackage{orcidlink}

\newcounter{obs}

\newcommand{\uddots}{\mathinner{
		\mkern1mu\raisebox{0.5em}{.}
		\mkern2mu\raisebox{1em}{.}
		\mkern2mu\raisebox{1.5em}{.}
}}
\newcommand{\uldots}{\mathinner{
		\mkern1mu\raisebox{1.5em}{.}
		\mkern2mu\raisebox{1em}{.}
		\mkern2mu\raisebox{0.5em}{.}
}}

\newtheorem{Theorem}{Theorem}%[section]
\newtheorem{Corollary}[Theorem]{Corollary}
\newtheorem{Lemma}[Theorem]{Lemma}
\newtheorem{Proposition}[Theorem]{Proposition}
\theoremstyle{Definition}
\newtheorem{Definition}[Theorem]{Definition}
\theoremstyle{remark}
\newtheorem{Remark}[Theorem]{Remark}
\newtheorem{Example}[Theorem]{Example}

\numberwithin{equation}{section}
\numberwithin{Theorem}{section}

\begin{document}

\title[Eccentric $p-$summing Lipschitz operators and integral inequalities on metric spaces and graphs]
{Eccentric $p-$summing Lipschitz operators and integral inequalities on metric spaces and graphs}

%----------Authors
\author[Arnau, R.]{R. Arnau\orcidlink{0000-0003-2544-8875}}
\address{%
	Instituto Universitario de Matem\'atica Pura y Aplicada\\
	Universitat Polit\`ecnica de Val\`encia\\
	46022 Valencia\\
	Spain}
\email{ararnnot@posgrado.upv.es (R.A.)}

\author[~S\'{a}nchez P\'{e}rez, E.A.]{E.\,A.~S\'{a}nchez P\'{e}rez\orcidlink{0000-0001-8854-3154}}
\address{%
	Instituto Universitario de Matem\'atica Pura y Aplicada\\
	Universitat Polit\`ecnica de Val\`encia\\
	46022 Valencia\\
	Spain}
\email{easancpe@mat.upv.es (E.A.S.P.)}

\author[Sanjuan, S.]{S. Sanjuan\orcidlink{0009-0001-5310-2559}}
\address{%
	Instituto Universitario de Matem\'atica Pura y Aplicada\\
	Universitat Polit\`ecnica de Val\`encia\\
	46022 Valencia\\
	Spain}
\email{ssansil@upvnet.upv.es (S.S.)}

%----------classification, keywords, date
%\subjclass{Primary 05C12; Secondary 05C10, 05C75}


\date{October 14, 2024}

\

\maketitle

%%%%%%%%%%%%%%%%%%%%%%%%%%%%%%%%%%%%%%%%%%%%%%%%%%%%%%%%

\begin{abstract}
The extension of the concept of $p-$summability for linear operators to the context of Lipschitz operators on metric spaces has been extensively studied in recent years. This research primarily uses the linearization of the metric space $M$ afforded by the associated Arens-Eells space, along with the duality between $M$ and the metric dual space $M^\#$ defined by the real-valued Lipschitz functions on $M.$ However, alternative approaches to measuring distances between sequences of elements of metric spaces (essentially involved in the definition of $p-$summability) exist. One approach involves considering specific subsets of the unit ball of $M^\#$ for computing the distances between sequences, such as the real Lipschitz functions derived from evaluating the difference of the values of the metric from two points to a fixed point.
We introduce new notions of summability for Lipschitz operators involving such functions, which are characterized by integral dominations for those operators. To show the applicability of our results, we use in the last part of the paper the theoretical tools obtained in the first part to the analysis of metric graphs. In particular, we show new results on the behavior of numerical indices defined on these graphs satisfying certain conditions of summability and symmetry.
\end{abstract}

%%%%%%%%%%%%%%%%%%%%%%%%%%%%%%%%%%%%%%%%%%%%%%%%%%%%%%%%

\section{Introduction}

The triangular inequality that satisfies every metric allows us to write another expression associated with any distance: if $(M,d)$ is a metric space, consider the expression
$$
\widehat{d}(x_1,x_2)= \sup_{y \in M} \big| d(x_1,y)-d(x_2,y) \big|, \quad x_1,x_2 \in M.
$$
Clearly, $\widehat{d}=d;$ indeed, by the triangular inequality, $\widehat{d} \le d,$ and for the other inequality just take $y=x_2.$ However, the situation changes if we compute the supremum on a smaller set $S \subseteq M,$ perhaps with some extra properties (e.g. compactness.) In this case, one can only be sure that the formula
$$
\widehat{d}_S(x_1,x_2)= \sup_{y \in S} \big| d(x_1,y)-d(x_2,y) \big|
$$
is a semi-norm, since it could happen
that $\widehat d(x_1,x_2)=0$ while $x_1 \ne x_2.$ For example, if $(M,d)$ is the Euclidean real line, $S=\{ 0\},$ $x_1=1/2$ and $x_2=-1/2.$ We will call such an expression $\widehat{d}_S$ a eccentric pseudometric associated to the distance $d.$

Eccentric pseudometrics as  $\widehat{d}_S$ have been used in \cite{arn} in the context of Lipschitz map representation to provide new definitions of summability for these maps. Some interesting questions could be raised about its own definition, such as the characterization of the subsets $S$ for which such a pseudometric $\widehat d_S$ is a metric, and when it is equivalent to $d.$ It can also be used to generate new tools for the study of the summability of Lipschitz operators, both in general metric spaces and in specific cases of special interest. We deal with this type of application in this paper, in which, after showing some new general results on summability of Lipschitz maps, we focus attention on the case of metric graphs.

Let us fix some concepts and notation. All the definitions and results needed on graph theory  are elementary; the reader can find information about in \cite{gods}. Throughout the paper we use standard  concepts  of Banach spaces, metric spaces and measure theory. A metric space $(M,d)$ is a set $M$ endowed with a function $d:M \times M \to \mathbb R^+$ satisfying that $d(x,y)=0$ if and only if $x=y,$ and for all $x,y,z \in M,$ $d(x,y)=d(y,x)$ and $d(x,z) \le d(x,y)+d(y,z)$ \cite{deza}. A pseudometric is defined in the same way, but $d(x,y)=0$ does not imply $x=y.$ A Lipschitz function $f:M \to N$ between metric spaces $(M,d_M)$ and $(N,d_N)$ is a function that satisfies for a certain constant $K>0$ that $d_N(f(x),f(y)) \le K d_M(x,y)$ for all $x,y \in M,$ and $Lip(f)$ is the infimum of all such constants $K$ \cite{cob,wilson1935}. If $X$ is a Banach space, we write $\| \cdot\|_X$ for its norm (or $\| \cdot\|$ if the space is fixed in the context), $B_X$ for its unit ball and $X^*$ for its dual space. For example, if $\mu$ is a regular Borel probability measure on the $\sigma-$algebra of Borel subsets of a metric space $M,$ we write $L^p(\mu)$ for the classical Banach lattice of all the $p-$integrable functions with respect to $\mu.$

\begin{Remark} \label{rem1}
	Let us place the pseudometrics $\widehat d_S$ defined above in the context of the classical linearization of metric spaces.
	Although we are working in the general framework of the metric spaces, for the case that $S=M,$ we can understand the formula $\widehat d$ for the computation of $d$ from the classical point of view of duality in Banach spaces.
	Indeed, if we fix a point $ z \in M $ in the metric space $ (M, d) $, we can consider an isometric embedding $ M \hookrightarrow AE(M) $. into the Arens-Eells function space (also called the free space) \cite{are,god}. The space is generated by real functions defined as the difference of characteristic functions $ \chi_{\{x\}} - \chi_{\{z\}} $. The dual space of this function space can be identified with the Banach space $ M^\# $, which consists of all real Lipschitz functions that are zero at $ z $.
	The reader can find a complete explanation of such spaces in \cite{god}. Take now the set of all Lipschitz functions as $x \mapsto f_y(x)=d(x,y)-d(y,z),$ that are clearly in the unit ball  and are $0$ in $z.$ Then, if $x_1,x_2 \in M,$ $\widehat d$ can be computed by  duality as 
	$$
	d(x_1,x_2)= \sup_{y \in M} | \langle \chi_{x_1} - \chi_{z} -\chi_{x_2} + \chi_z, f_y \rangle | =  \sup_{y \in M} | d(x_1,y)-d(x_2,y) | = \widehat d(x_1,x_2).
	$$
	Indeed, note that it is enough to take $y=x_2$ to attain the value, and by the triangular inequality we always have that $ | d(x_1,y)-d(x_2,y) | \le d(x_1,x_2)$ for every $y \in M.$ 
	
	In other words, the functions $\{f_y:y \in M\} \subset M^\#=(AE(M))^*$ are norming for the elements  $\chi_{\{x_1\}}- \chi_{\{x_2\}}$ (but not necessarily for all the elements of $AE(M)$). To see a counterexample, just consider the metric space $\{X_1,X_2,X_3,X_4\}$ endowed with the discrete metric. Then a direct computation shows that 
	$$
	\sup_{f_{X_i} \, :\, i=1,2,3,4} \Big\langle \frac{1}{2} (\chi_{\{X_1\}}-\chi_{\{X_2\}}) + \frac{1}{2} (\chi_{\{X_3\}} - \chi_{\{X_4\}}), f_{X_i} \Big\rangle = 1/2,
	$$
	while 
	$$ 
	\Big\|  \frac{1}{2} (\chi_{\{X_1\}}-\chi_{\{X_2\}}) + \frac{1}{2} (\chi_{\{X_3\}} - \chi_{\{X_4\}})  \Big\|=1.
	$$
\end{Remark}

Let us explain our approach. To do so, let us recall some fundamental issues about summability in Banach spaces. After the introduction of the powerful machinery of so-called operator ideals \cite{deflo,piets}, the central notions of summability in Banach spaces are often characterized by means of integral inequalities for operators. For example, the relation between weak convergence and norm convergence of sequences is approached from the point of view of the properties of the operators that transform weakly summable sequences into strongly summable sequences by defining the so-called $1-$summing operators \cite{deflo,djt,piets}. Recall that, given $p \ge 1$ and a sequence $(x_i)_i$ in a Banach space $(X,\| \cdot\|)$, two sequence norms can be naturally defined for it, that are the so called strong $p-$norm $\| \cdot\|_{p,s}$ and the weak $p-$norm $\| \cdot\|_{p,w}$, and are given by
$$
\|(x_i)_i\|_{p,s} = \Big( \sum_{i=1}^\infty \| x_i\|^p \Big)^{1/p}, \quad \text{and} \quad  \|(x_i)_i\|_{p,w} =\sup_{x' \in B_{X^*}} \Big( \sum_{i=1}^\infty \big| \langle x_i, x' \rangle \big|^p  \Big)^{1/p},
$$
respectively. So, a sequence is absolutely $p-$summable (also said strongly $p-$summable) when its strong $p-$norm is finite, and weakly $p-$summable if its weak norm is.
A linear operator between Banach spaces $T:X \to Y$ is absolutely $p-$summing if and only if it transforms weakly $p-$summable sequences into strong $p-$summable sequences, which is equivalent to the existence of a constant $K>0$ such that the inequality
$$
\|(T(x_i))_{i=1}^n\|_{p,s} \le K  \|(x_i)_{i=1}^n \|_{p,w}
$$
holds for every finite sequence in $X.$ Pietsch's factorization theorem \cite[Ch.1]{djt} establishes that this is equivalent to the existence of a Borel regular measure $\mu$ on $B_{X^*}$ such that the following integral domination
$$
\| T(x) \|_Y \le K \, \int_{B_{X^*}} \big| \langle x, x' \rangle \big| \, d \mu(x')
$$
holds for all $x \in X.$
For example, using the associated factorization properties for these ideal operators, it can be easily shown that a Banach space in which weakly summable sequences and strongly summable sequences coincide is necessarily finite dimensional (see \cite[Ch.1]{djt}).

Mimicking (to some extent) the situation in Banach spaces, in this paper we also face the summability problem based on integral inequalities as above but for the relation between the summability associated to the distance and the summability associated to the eccentric distance, which characterize some relevant properties regarding the summability of Lipschitz functions in general metric spaces. Let us first clarify that the translation of the operator-ideal way of understanding summability from the linear context to the Lipschitz setting has already been deeply studied in recent years \cite{ach1,ach2,cha,chen,farmer}. This research has been even extended to a broader class of non-linear operators \cite{ang}. The main definitions and results have already been adapted to this context, and some relevant theorems have been proved according to this Lipschitz version of ideal operator theory (see, e.g., \cite{ach1,farmer,fer}). Our purpose in the present paper does not follow this research trend, but the comparison of the sequence summability properties associated to the original $d$-metric and the associated $\widehat d$-metric corresponding to the adapted versions of the strong and weak norms in the new context. Let us now introduce these notions; we have not found similar definitions in the vast literature on metric spaces, so we present them as new.

Fix $p \geq 1.$ We say that two  sequences $(x^1_i)_i$ and $(x^2_i)_i$ in a metric space $(M,d)$ are $p-$absolutely close to each other if
$$
D_{p,ac} \big((x^1_i)_i,(x^2_i)_i \big) := \Big( \sum_{i=1}^\infty d(x^1_i,x^2_i)^p \Big)^{1/p} < \infty.
$$
In relation with the action of the elements of the dual space, we say that these two sequences are $p-$weakly close to each other if
$$
D_{p,wc} \big((x^1_i)_i,(x^2_i)_i \big) = \sup_{f \in B_{M^\#}} \Big( \sum_{i=1}^\infty \big| f(x^1_i)- f(x^2_i) \big|^p  \Big)^{1/p}< \infty.
$$
and that they are $p-$eccentrically close to each other if
$$
D_{p,cc} \big((x^1_i)_i,(x^2_i)_i \big) = \sup_{y \in M} \Big( \sum_{i=1}^\infty \big| d(x^1_i,y)- d(x^2_i,y) \big|^p  \Big)^{1/p}< \infty.
$$
Note that the last definition can be directly adapted to the case in which the supremum is computed over a subset $S \subseteq M;$ we will write $D^S_{p,cc}$ in this case.   Note also that, after Remark \ref{rem1}, the definition of $D^S_{p,cc}$ can be given also by restricting the set in which the supremum in $D_{p,wc}$ is computed; that is, $D^S_{p,cc} \le D_{p,wc} $ for every $S \subseteq M.$ However, we prefer to keep the different  notations $D_{p,cc}$ and $D_{p,wc}$  for clarity, due to the relevance of the sets defined by functions as the $f_y \in M^\#$ appearing in this remark.

We will analyze the fundamental relations between these three ways of understanding proximity between sequences in the next section, as well as the properties of the following two classes of operators, the study of which will be the main objective of this work.

%CUIDADOOOO LA SIGUIENTE DEFINICION ESTA YA DICHA Y ESTUDIADA EN EL PAPER
%Arnau, Roger, Jose M. Calabuig, and Enrique A. Sánchez Pérez. "Representation of Lipschitz Maps and Metric Coordinate Systems." Mathematics 10.20 (2022): 3867.
%
%ALLI SE LLAMA METRIC SUMMING     

Extending the definition of metric summing operator given in \cite[Def.4]{arn}, if we fix $p \geq 1$ and a subset $S \subseteq M,$ we say that a map $T:M \to N$ between metric spaces $(M,d_M)$ and $(N,d_N)$ is 
eccentric $p-$summing if there is a constant $K>0$ such that for every couple of finite sequences  $(x^1_i)_{i=1}^n$ and $(x^2_i)_{i=1}^n$ in the metric space $(M,d_M)$ we have
$$
\Big(  \sum_{i=1}^n d_N(T(x^1_i),T(x^2_i))^p \Big)^{1/p} \le K \,
\sup_{y \in S} \Big( \sum_{i=1}^n \big| d_M(x^1_i,y)- d_M(x^2_i,y) \big|^p  \Big)^{1/p},
$$
that is, $ D_{p,ac} \big((T(x^1))_{i=1}^n,(T(x^2_i))_{i=1}^n) \big) \le K \,  D_{p, cc} \big((x^1_i)_{i=1}^n,(x^2_i)_{i=1}^n) \big).$ 
%Note that this definition was made only for $p=1$ in \cite{arn}, but the generalization to the case $p>1$ is straightforward. 
%If we want to be more precise in the notation, we will say that such an operator is $S-$eccentric $p-$summing.  
This definition was made in \cite{arn} for the case that $S$ is a metric generating system, but it makes sense for any subset $S \subseteq M.$
The infimum of such constants $K$ is denoted by $MLip_p(T).$ The reader who is familiar with the ideals of linear operators and their adaptation to the case of Lipschitz maps can easily understand that this definition is inspired by the case of $1-$summing operators (see for example \cite[Ch.1]{djt}).  In fact, as said in \cite{arn}, all eccentric $1-$summing operators are Lipschitz $1-$summing in the sense of Farmer and Johnson \cite{farmer}, as can be easily seen taking into account that for any fixed $y \in M,$ the function $x \mapsto d(x,y)$ is a map belonging to the unit ball of the Banach space $M^\#$ of all real Lipschitz functions, following the notation used in \cite{farmer}.

The next definition is inspired to some extend in other classical operator ideal, the one of the $(p,q)-$mixing operators (see \cite[Ch.21]{deflo}), which has also found its analogue in the case of the Lipschitz operators \cite{cha2}. Thus, in a similar way, we say that an operator $T:M \to N$ is eccentrically $p-$approximating if  for every couple of finite sequences  $(x^1_i)_{i=1}^n$ and $(x^2_i)_{i=1}^n$ in the metric space $(M,d_M)$ we have
$$
\sup_{z \in N} \Big( \sum_{i=1}^n \big|
d_N(T(x^1_i), z) - d_N(T(x^2_i),z) \big|^p \Big)^{1/p} \le K \,
\sup_{y \in M} \Big( \sum_{i=1}^n \big| d_M(x^1_i,y)- d_M(x^2_i,y) \big|^p   \Big)^{1/p},
$$
that is, $ D_{p,cc} \big((T(x^1))_{i=1}^n,(T(x^2_i))_{i=1}^n) \big) \le K \,  D_{p,cc} \big((x^1_i)_{i=1}^n,(x^2_i)_{i=1}^n) \big).$ 
The infimum of such constants $K$ is $ELip_{p}(T).$ 

Both definitions can be directly adapted if we use eccentric metrics defined by subsets $S $ of $M$ and $B $ of $N$ instead of the direct ones, that is, with the supremum in the formulas computed over $S$ and $B$ instead of $M$ and $N,$ respectively.
%If so, we will say that the operators are $S-$eccentric $p-$summing and $S,B-$eccentrically approximating, and will use the super-indices $S$ and $B$ in all the related notations.

\section{Absolutely close and eccentrically close sequences in metric spaces} \label{S2}

In this section we show the fundamental relationships between the various definitions of $p-$summation of sequences in metric spaces that we have introduced above. As we will see in the next section, the operators we define fulfill that they convert pairs of sequences that are similar with respect to one summation method into sequences that are similar with respect to another, stronger summation method. To understand this, we first need to know what is the natural order between these sums, to distinguish between stronger and weaker summation methods.

Given two sequences $(x^1_i)_{i=1}^n$ and $(x^2_i)_{i=1}^n$ in a metric space $(M,d),$
the definitions of absolutely close and eccentrically close are directly inspired by those of absolutely and weakly summable sequences. However, while the former leads to the corresponding notion in the theory of Banach spaces (just change $d(x^1,x^2)$ to $\||x^1-x^2\|$ in the definition), the concept of eccentrically close sequences does not coincide with the notion of weakly summable sequences. Let us fix the basic relations in the following result, and the relation to the case of Banach space in the following.

\begin{Proposition}
	Let  $(M,d)$ be a metric space and $p \ge 1,$ and
	consider two sequences $(x^1_i)_{i=1}^\infty$ and $(x^2_i)_{i=1}^\infty$ in it. Then for every $S \subseteq M,$
	$$
	D^S_{p,cc} \big((x^1_i)_i,(x^2_i)_i \big) \le  D_{p,cc} \big((x^1_i)_i,(x^2_i)_i \big) \le D_{p,wc} \big((x^1_i)_i,(x^2_i)_i \big)  \le D_{p,ac} \big((x^1_i)_i,(x^2_i)_i \big).
	$$
\end{Proposition}
\begin{proof}
	The first inequality is obvious. 
	After Remark \ref{rem1}, we know that all the functions  of the set $\{f_y : y \in S\}$ defined there belong to the unit ball $B_{M^\#},$ what makes the inequality in the middle of the expression obvious. The inequality in the right hand side is just a consequence of the duality of $AE(M)$ and $M^\#= Lip_0(M).$ Indeed,  for every $x^1,x^2,y \in M$ and $f \in B_{M^\#}$ we have that
	$
	| f(x^1_i)- f(x^2)| \le d(x^1_i,x^2_i), \quad i \in \mathbb N,
	$
	and so for all $y \in M^\#,$
	$$
	\sum_{i=1}^\infty \big| d(x^1_i,y)- d(x^2_i,y) \big|^p  \le  \sum_{i=1}^\infty  d(x^1_i,x^2_i)^p,
	$$
	which gives the result. 
\end{proof}

After these obvious relations, we show in what follows conditions under which the converse inequalities could hold. This is relevant because we will determine then metric spaces for which the identity map has a certain summing property, what will provide some equivalent integral average formulas for the original metric. Let us see first what happens for the linear case. In what follows, we will assume that $p$ is a fixes parameter $1 \le p < \infty.$

\begin{Remark}
	Let $(X, \| \cdot\|)$ be a Banach space. Recall that a subset $S \subseteq B_{X^*}$ is said to be norming if it is enough for computing the norm for all the elements of $X$ by duality. That is, the formula
	$$
	\sup_{x' \in S} |\langle x, x' \rangle| =\|x\|
	$$
	holds for every $x \in X.$ It is well-known (in fact it is a consequence of a direct computation), that for any sequence $(x_i)_i,$ its weak $p-$norm can be computed just using a norming subset $S$ instead of all $B_{X^*}$ (see for example \cite[Ch.2, \S 3, p. 36]{djt}). This fact has also consequences in the Lipschitz case, since in the computation of  $D_{p,wc} \big((x^1_i)_i,(x^2_i)_i \big), $ the set $B_{M^\#}$ can be substituted by any norming subset: this is just a consequence of  understanding  this expression as a dual action between the space $AE(M)$ and its dual space $M^\#.$  However, the set of all the functions defined in Remark \ref{rem1} $\{f_y : y \in M\} \subseteq B_{M^\#}$  is  in principle not norming for the whole space $AE(M)$ as shown there,  so this argument cannot be used, although of course it is norming for the functions $\chi_{\{x_1\}} - \chi_{\{x_2\}}.$  
\end{Remark}

In order to facilitate the finding of applications (for example, for finite graphs, see Section \ref{sec4}), we extend the notion of norming to the following: we say that a subset $S \subseteq B_{X^*}$ of a Banach space $(X,\| \cdot\|)$ is  $k-$norming for short if there is a constant $k>0$ such that 
$$
\|x\| \le k \, \sup_{x' \in S} |\langle x, x' \rangle |, \quad x \in X.
$$
(Of course, the inequality $ \sup_{x' \in S} |\langle x, x' \rangle | \le \|x\|$ always holds.)

\begin{Proposition}
	Let $(M,d)$ be a  pointed metric space with distinguished point $0$, and
	consider two sequences $(x^1_i)_{i=1}^n$ and $(x^2_i)_{i=1}^n$ in it. Then, if the set $\mathcal F_S=\{f_y = d(\cdot,y) -d(y,0): y \in S \subset M\} \subseteq B_{M^\#}$ is $k-$norming for a certain constant $k,$ we have
	$$
	D^S_{p,cc}((x^1_i)_{i=1}^n,(x^2_i)_{i=1}^n) 
	\le D_{p,wc}((x^1_i)_{i=1}^n,(x^2_i)_{i=1}^n) 
	\le k \, D^S_{p,cc}(x^1_i)_{i=1}^n,(x^2_i)_{i=1}^n).
	$$ 
\end{Proposition}
\begin{proof}
	It is just a consequence of writing the summation expressions in terms of duality.
	Indeed, note that the first inequality is obvious. For the second one, just write $p'$ for the extended real number that satisfies $1/p+1/p'=1,$ $\ell_{p'}$ the classical space of sequences of real numbers with finite $p'$ with the corresponding norm, and notice that
	$$
	D_{p,wc}((x^1_i)_{i=1}^n,(x^2_i)_{i=1}^n) =
	\sup_{f \in B_{M^\#}} \,\, 
	\sup_{(\lambda_i)_{i=1}^n \in B_{\ell^{p'}}}
	\Big\langle \sum_{i=1}^n \lambda_i \big(\chi_{\{x_i^1\}} - \chi_{\{x_i^2\}} \big), f \Big\rangle 
	$$
	$$
	\le k \, \sup_{f_y \in \mathcal F_S} \,\,\, \sup_{(\lambda_i)_{i=1}^n \in B_{\ell^{p'}}}
	\Big\langle \sum_{i=1} ^n \lambda_i \big(\chi_{\{x_i^1\}} - \chi_{\{x_i^2\}} \big), f_y \Big\rangle = k\,  D^S_{p,cc}((x^1_i)_{i=1}^n,(x^2_i)_{i=1}^n).
	$$
\end{proof}

\vspace{0.5cm}

\section{Factorization of eccentric $p-$summing and eccentrically $p-$approximating Lipschitz operators } \label{secfac}

In this section we show the main characterizations of the classes of operators that we defined in the introduction, as well as the properties that can be derived from them. As will be seen, the definition of these operators involves the uniform transformation of pairs of sequences that are summable with respect to a method of those explained in Section \ref{S2}, to a pair of sequences that are summable with respect to a stronger method of those also appearing in Section \ref{S2}. The characterizations of the resulting operators are given mainly by integral inequalities, which can be translated into factorization theorems for the Lipschitz operators involved.

The techniques used here are in some sense related to those employed in \cite{rodsan}, that analyze suitable formulas to represent metrics as integral averages as the ones involved in the so called Wasserstein space \cite[\S 2.1]{pan}. Also, a particular case of eccentric $1-$summing operators were studied in \cite{arn} (metric summing operators in this reference). 
For the case of  Banach spaces, one of the first questions that were faced in the original paper on Lipschitz $p-$summing operators \cite{farmer} was connected with the problem of what happens if the Lipschitz  map involved is also linear. It was proved there that 
in this case, to be Lipschitz $p-$summing and being absolutely  $p-$summing was the same \cite[Th.2]{farmer}.  In this section we compare also different types of summability for  Lipschitz operators, and we show some structural applications for general metric spaces. In the next section we will analyze the concrete case of infinite metric graphs as an example of application, and also as a proposal of new tool for studying symmetry notions for graphs. The next theorem, that is the consequence of a classical separation argument that appears often in results on factorization of operators, is the main result of this section.

\begin{Remark}
	Note that the requirements on the operators that appear in all the results of this section imply that they are in particular Lipschitz. This is the reason why we do not ask in the statements of the propositions that the operators are Lipschitz, we simply write that they are maps. The proofs that they are all Lipschitz are always consequences of direct computations.
\end{Remark}

%$$
%\xymatrix{
	%L^{2}[- \pi, \pi] \ar[rr]^{\mathcal G_2} \ar@{.>}[d]_{i} & & \ell^{2}\\
	%L^{p}[- \pi, \pi] \ar@{.>}[rr]^{\mathcal F_p}&  & \ell^{p'}  \ar@{.>}[u]_{M_\lambda}}
%$$
%
%$$
%\xymatrix{
	%X_1(\mu) \ \ar[r]^{T} \ar@{.>}[d]_{M_h} & \ Y_1(\nu) \ \ar@{.>}[r]^{i}
	%& \ Y_1(\nu)'' \ \ar@{.>}[r]^{\eta} & \ Y_1(\nu)'^* \\
	%X_2(\mu) \ \ar[r]^{S} & \ Y_2(\nu) \ \ar@{.>}[rru]_{R_{\xi^*}}}
%$$
%
%$$
%\xymatrix{
	%M \ \ar[r]^{T} \ar@{.>}[d]_{M_h} & \ Y_1(\nu) \ \ar@{.>}[r]^{i}
	%& \ \mathbb R \\
	% \ \ar@{.>}[rru]_{R_{\xi^*}}}
%$$

Let us write now the main characterization of the domination of Lipschitz functionals in terms of the summability properties that they satisfy. In other words, the class of functionals that are characterized below (eccentric $p-$summing functionals) carry $p-$eccentrically close couples of (finite) sequences in a metric space to $p-$absolutely close couples of sequences in a uniform way, and the uniform constant relating the corresponding summations is the operator eccentric $p-$summing norm $L_p(T).$

Let us remark that, for passing from a summability property of a map involving a eccentric $p-$pseudometric to an integral domination we need the set used to compute it (the set that we denote by $S \subseteq M$ in the previous sections) to be compact for the metric topology of $(M,d).$ To underline this fact, we use in the following results the letter $K$ (which usually denotes compact sets) for $S.$

\begin{Theorem} \label{pietfunc}
	Let $(M,d)$ be a metric space and $K$ a compact subset of $M$. 
	Let  $ p \ge 1,$ $f: M \to \mathbb R$ be a mapping and $C > 0.$ Then the  following statements are equivalent.
	\begin{enumerate}
		\item[(i)] For any $n \in \mathbb N$, $x_1, x_2, \ldots, x_n, y_1, y_2, \ldots, y_n \in M$,
		\begin{equation*}
			\Big( \sum_{i=1}^n |f(x_i)-f(y_i)|^p\Big)^{1/p} \leq C \sup_{w \in K} \Big(  \sum_{i=1}^n | d(x_i, w) - d(y_i, w) |^p \Big)^{1/p}.
		\end{equation*}
		\item[(ii)] There exists a Borel regular probability measure $\mu$ on $K$ such that for any $x, y \in M$
		\begin{equation*}
			|f( x)- f(y) | \leq C \Big( \int_K | d(x,w) - d(y,w) |^p \,  d\mu(w) \Big)^{1/p}.
		\end{equation*}
	\end{enumerate}
	A real function satisfying (i) or (ii) is Lipschitz with constant $ \le C.$ Following the general notation introduced before, we call such a function a eccentric $p-$summing functional.
\end{Theorem}

\begin{proof}  The proof is an application of a classical separation argument that is found in many factorization theorems. We will use here Ky Fan's Lemma instead of the argument based on the Hahn-Banach Theorem that is usual. 
	
	$(i) \Rightarrow (ii)$ First note that, taking into account that the same elements $x_i,y_i$ could appear several times in the inequalities in \textit{(i)}, using rational number approximation we can see that these inequalities can be extended to a larger class of inequalities involving also positive constants such as the following, where $x_1, x_2, \ldots, x_n, y_1, y_2, \ldots, y_n \in M$ are general finite sequences of elements of $M$ and $a_1, a_2 \ldots, a_n > 0,$
	\begin{equation} \label{eqenproof}
		\sum_{i=1}^n a_i | f( x_i)- f( y_i)|^p \leq C \sup_{w \in K} \sum_{i=1}^n a_i | d(x_i, w) - d(y_i, w) |^p .
	\end{equation}
	The interested reader can find an explanation of this argument in \cite{farmer}).
	
	Let us recall Ky Fan's minimax Lemma (the one used in analysis, there is another famous Ky Fan's Lemma on  labellings of triangulations that generalizes Tucher's Lemma). 
	In particular, a family $\mathcal F$ of real functions is concave if every (finite) convex combination of elements of the family gives again an element of $\mathcal F.$
	Ky Fan's Lemma (see for example  \cite[Ch.9]{djt}) estates 
	that,
	\begin{itemize}
		\item if $W$ is a compact convex subset of a Hausdorff topological vector
		space and 
		
		\item  $\Psi$ is a concave family of lower semi-continuous, convex, $\mathbb R-$valued
		functions on $W,$ 
		
		\item and  there is $c \in \mathbb R$ satisfying that  for every $\psi \in \Psi,$ there exists $x_\psi \in  W
		$ such that $\psi(x_\psi) \le c,$ 
	\end{itemize}
	then there exists $x \in W$ such that $\psi(x) \le c$ for every $\psi \in \Psi.$

	Let us show the proof. Fix $f$ be a real function as the one in the statement of the theorem. Since the set $K$ is compact in $M,$ we can define the space of continuous functions on it, $C(K),$ which dual space is the space of Borel regular measures $\mathcal M(K).$
	Define  for any subset of $M \times M \times [0, + \infty[$ as $A = \big\{ (x_1, y_1, a_1), (x_2, y_2, a_2), \ldots, (x_n, y_n, a_n) \big\} \subseteq M \times M \times [0, + \infty[,$ $n \in \mathbb N,$ the function $\psi_A : \mathcal M(K)  \to \mathbb R $ given by 
	$$
	\psi_A(\mu)=   \sum_{i=1}^n a_i \, |f(x_i)- f( y_i)|^p - \int_{K} \sum_{i=1}^n a_i \, C^p | d(x_i,w) - d(y_i,w) |^p  \, d \mu(w) .
	$$	
	Consider the dual pair $\langle  C(K), \mathcal M(K)\rangle,$ and consider the space $\mathcal M(K)$ endowed with its weak* topology. The function $\psi_A$ is clearly continuous, since for every fixed $x \in M,$ the function $w \mapsto d(x,w)$ is Lipschitz continuous, so in particular the $p-$th power appearing in the function belongs to $C(K).$ The integral represents the dual action, so it works by duality and is continuous for the weak* topology of $ \mathcal M(K).$ Due to the multiplication by arbitrary scalars that appears in the definition, the family of all these functions is concave, and each of them is also convex. 
	
	On the other hand, there is always a measure $\mu_A$ that satisfies that $\psi_A(\mu_A) \le 0.$ Indeed, by compactness the supremum appearing in the corresponding inequality given in Equation (\ref{eqenproof}) is always attained at a certain point $w_A,$ so it is enough to consider $\mu_A= \delta_{\{w_A\}}$ to get this requirement. Now, we apply Ky Fan's Lemma for $c=0$ to obtain a fixed  measure $\mu$ such that
	$$
	\sum_{i=1}^n a_i \, |f(x_i)- f( y_i)|^p \le \int_{K} \sum_{i=1}^n a_i \, C^p | d(x_i,w) - d(y_i,w) |^p  \, d \mu(w)
	$$
	for all the functions of the family. In particular, for every $x,y \in M,$
	$$
	\big| f(x)- f(y) \big|^p \le C^p \, \int_K   | d(x_i,w) - d(y_i,w) |^p  \, d \mu(w).
	$$

	$(ii) \Rightarrow (i)$
	If there is probability measure $\mu$ as in $ (ii),$ taking into account that $\mu$ is a probability measure and computing the power $p$ in the inequality we get
	\begin{align*}
		\sum_{i=1}^n  | f(x_i)- f(y_i) |^p
		& \leq C^p \, \int_K \left( \sum_{i=1}^n | d(x_i,w) - d(y_i,w) |^p \right) d\mu(w) \\
		& \leq  C^p \,  \sup_{t \in K} \sum_{i=1}^n | d(x_i,t) - d(y_i,t) |^p,
	\end{align*}
	so $(i)$ is obtained.

\end{proof}

This result can be written also as a commutative factorization  for the Lipschitz functional $f$ through a subset of $L^p(\mu),$ as 
$$
\xymatrix{
	M \ar@{.>}[d]_{i_d}  \ \ar@{.>}[r]^{f} & \ \mathbb R \,\, ,\\
	L^p(\mu) \supset S \,\,\,\,\,\,\,\,\,\,\,\,\,\,\,\,\,\,\,\, \ \ar@{.>}[ru]_{\widehat f}}
$$
where $i_d$ is a Lipschitz map given by $i_d(x)(\cdot)= d( x,\cdot) \in L^p(\mu),$ $S$ is the set of all the classes of  functions as $d(x,\cdot)$ that are equal $\mu-$a.e.  and $\widehat f$ is also a real-valued Lipschitz functional on $S$ given by $\widehat f(i_d(x))=f(x).$ Indeed, note that for $x,y \in M,$ we have that the difference of the functions $w \mapsto i_d(x)$ and $w \mapsto i_d(y)$ satisfies
$$
\Big( \int_K | d(x,w) - d(y,w) |^p \,  d\mu(w) \Big)^{1/p} 
$$
$$
\le  \Big( \int_K \sup_{w \in K} | d(x,w) - d(y,w) |^p \,  d\mu(w) \Big)^{1/p} \le \mu(K) \, \sup_{w \in K} | i(x)(w) - i(y)(w) | \le d(x,y),
$$ 
and the continuity of $\widehat f$ is assured by the inequality in Theorem \ref{pietfunc} (2). The functions $i_d$ and $\widehat f$ are Lipschitz.

\vspace{0.5cm}

\begin{Remark} \label{remp}
	It can be easily seen that the same proof of Theorem \ref{pietfunc}  works for eccentric $p-$summing operators if the 
	differences $|f(x_i)-f(y_i)|$ are substituted by $d(x_i,y_i)$ in the left hand side of the inequality in the statement of this theorem. Recall that an operator $T: M \to N$ is eccentric $p-$summing for a compact set $K \subset M$ if there is a constant $R>0$
	such that for every couple of finite sequences  $(x^1_i)_{i=1}^n$ and $(x^2_i)_{i=1}^n$ in the metric space $(M,d_M)$ we have
	$$
	\Big(  \sum_{i=1}^n d_N(T(x^1_i),T(x^2_i))^p \Big)^{1/p} \le R \,
	\sup_{w \in K} \Big( \sum_{i=1}^n \big| d_M(x^1_i,w)- d_M(x^2_i,w) \big|^p  \Big)^{1/p},
	$$
	or, if we write it in terms of sequence metrics, 
	$$
	D_{p,ac} \big((T(x^1))_{i=1}^n,(T(x^2_i))_{i=1}^n) \big) \le R \cdot  D_{p, cc} \big((x^1_i)_{i=1}^n,(x^2_i)_{i=1}^n) \big).$$ 
	
	Following the arguments in the proof of Theorem \ref{pietfunc}, we find that this is equivalent to the integral domination
	$$
	d(x,y) \le R \, \Big( \int_K | d(x,w) - d(y,w) |^p d\mu_f (w) \Big)^{1/p} 
	$$
	for all $x,y \in K.$ This provides an extension of Theorem 2 in \cite{arn} for the $p-$th power case.
\end{Remark}

The previous results can also be written when a weak-type metric for sequences is put in the left hand side of the inequality. For example, if we consider a mapping $T : M \to N$ and $f$ belongs to a certain subset $D$ of Lipschitz functions on $N,$ the set of compositions $\{ f \circ T: f \in D\}$ gives the following application of Theorem \ref{pietfunc}. Note that the supremum in the left can also be read as: for every $f \in D,$ the inequality in Theorem \ref{pietfunc} works for $f \circ T.$

\begin{Corollary} \label{pietsup}
	Let $(M,d), (N,\rho)$ be two metric spaces and $K$ a compact subset of $M$.
	Let $T : M \to N$ a mapping and $C > 0,$ and let $D \subseteq Lip_0(N).$ The following statements are equivalent.
	\begin{enumerate}
		\item For any $n \in \mathbb N$, $x_1, x_2, \ldots, x_n, y_1, y_2, \ldots, y_n \in M$,
		\begin{equation*}
			\sup_{f \in D} \Big( \sum_{i=1}^n |f \circ T (x_i) - f \circ T (y_i) |^p \Big)^{1/p}  \leq C \sup_{w \in K} \Big( \sum_{i=1}^n | d(x_i, w) - d(y_i, w) |^p \Big)^{1/p} .
		\end{equation*}
		\item For every $f \in D$  there exists a Borel regular probability measure $\mu_f$ on $K$ such that for any $x, y \in M$
		\begin{equation*}
			|f \circ T (x) -  f \circ T (y) |  \leq C  \Big( \int_K | d(x,w) - d(y,w) |^p d\mu_f (w) \Big)^{1/p} .
		\end{equation*}
	\end{enumerate}
	That is, if $T$ satisfies any of these statements, then it is eccentrically $p-$approximating.
\end{Corollary}

\vspace{0.5cm}

This result can be also written in terms of factorizations. If $T : M \to N$ is a mapping, we can draw the following scheme for a given real functional $f,$
$$
\xymatrix{
	M\ar[rr]^{T} \ar@{.>}[d]_{i_d} & & N  \ar@{.>}[d]_{f}\\
	S  \ar@{.>}[d]_{i}  \ar@{.>}[rr]^{\widehat f}&  & \mathbb R \\
	L^p(\mu) & &  }
$$
where  $i$ is a inclusion map and  $i_d$ is the Lipschitz map given by $i_d(x)(\cdot)= d( x,\cdot) \in L^p(\mu)$ that was explained in the factorization associated to Theorem \ref{pietfunc}; also the set    $S$ was explained there, as well as $\widehat f.$  It can be easily seen that all the maps in this commutative factorization diagram are Lipschitz.

\vspace{0.5cm}

\begin{Remark}
	In the case that we use the adaptation of this notion for general Lipschitz maps $T:M \to N,$ we find the definition of eccentrically $p-$approximating operator. Assume that $(M, d_M)$ and $(N, d_N)$ are metric spaces, and $T:M \to N$ is a map. Let us recall the definition of eccentrically $p-$approximating operators associated to a pair of compact subsets $K_1$ and $K_2.$
	Suppose that the following inequalities are satisfied for a constant $Q>0$ and every finite sequences of elements 
	$x^1_1,...,x^1_n$ and $x^2_1,...,x^2_n$ in $M$
	$$
	\sup_{y \in K_2} \Big(
	\sum_{i=1}^n\big| d_N(T(x^1_i),y)-d_N(T(x^2_i),y) \big|^p \Big)^{1/p}
	\le
	Q \, \sup_{x \in K_1} \Big(
	\sum_{i=1}^n  \big| d_M(x^1_i,x)-d_M(x^2_i,x) \big|^p \Big)^{1/p},
	$$
	where $K_1 \subset M$ and $K_2  \subset N$ are compact subsets.

	Note now that for every probability measure $\mu_N \in \mathcal M(K_2),$ the inequality above clearly gives the domination 
	$$
	\sum_{i=1}^n  \int_N  \big| d_N(T(x^1_i),y)-d_N(T(x^2_i),y) \big|^p \, d \mu_N(y) \le Q^p
	\, \sup_{x \in K_1} \Big(
	\sum_{i=1}^n  \big| d_M(x^1_i,x)-d_M(x^2_i,x) \big|^p \Big)^{1/p}.
	$$
	
	The same argument that proves Theorem \ref{pietfunc} writing 
	$$
	\left( \int_N  | d_N(T(x_1),y)-d_N(T(x_2),y) |^p \, d \mu_N(y) \, \right)^{1/p}
	$$ 
	instead of $|f(x^1_i)- f(x^2_i)|$
	shows that there is a probability measure $\mu_M  \in \mathcal M(K_1)$ such that
	$$
	\int_N  \big| d_N(T(x_1),y)-d_N(T(x_2),y) \big|^p \, d \mu_N(y) \le Q^p
	\int_M  \big| d_M(x_1,x)-d_M(x_2,x) \big|^p \, d \mu_M(x)
	$$
	for all $x_1,x_2 \in M.$ This gives the desired characterization for eccentrically $p-$approximating operators by means of integral inequalities.
\end{Remark}

\vspace{0.5cm}

\section{Applications: distance summing inequalities in metric graphs}  \label{sec4}

Having developed the core theoretical framework for eccentric $p-$summing domination, we move now to apply these results within the context of metric undirected graphs. All the notions and results that we use on graphs are elementary. The reader can find all the information needed about graphs and distances in graphs for example in \cite{god} and \cite{buck,deza}, respectively.

Consider a weighted connected undirected graph with a countable set of vertices endowed with the weighted shortest path distance $q.$  We need first to clarify the type of  graphs we will deal with. We will consider the case in which, although there is an infinite set of vertices in the graph, all of them are connected by a finite path. This finiteness of the paths is relevant, since given a certain $p \ge 1,$  we can use the usual formula for the $p-$weighted shortest path distance (positive weights) without the need of considering infinite paths. That is, given a couple of vertices $v_1$ and $v_2$ in the graphs, there is a finite number of paths connecting them, so the shortest path $p-$distance can be defined as usual by
$$
q_p(v_1,v_2)= \inf  \Big\{  \Big( \sum_{i=1}^{n-1} w(v^i, v^{i+1})^p \Big)^{1/p} : \text{the sequence $(v^i)_i$ connects $v^1=v_1$ and $v^n=v_2$} \Big\},
$$
where $w(v^i,v^{i+1})$ denotes the weight of the edge connecting both vertices.

\begin{Example} \label{exfin1}
	A canonical example of the type of structure we have fixed is a graph composed of vertices that are identified with the elements of an increasing convergent sequence in $\mathbb R$ with its boundary $s$ (which does not belong to the sequence), in which the connectivity is given by the adjacency on the real line and the distance (weight) between consecutive vertices is the size of the segment connecting them. Furthermore, the boundary vertex $s$ is assumed to be connected to all other vertices, and the weights for the distance are again the size of the segment connecting it to the other points in the sequence.
	
	It can be easily seen that the distance of the shortest path between any two vertices is in this case again the size of the segment connecting the two vertices. This (infinite) graph is clearly compact, since it is isomorphic and isometric to the convergent sequence with its limit. %We are thinking for example of sequences like $(1-1/2^i)_i$ with its limit, $1.$
	\begin{figure}[ht]
		\centering
		\begin{tikzpicture}
			\node[draw, circle] (n1) at (-6, 0) {$v_0$};
			\node[draw, circle] (n2) at (0, 0) {$v_1$};
			\node[draw, circle] (n3) at (3, 0) {$v_2$};
			\node[draw=none] (n4) at (4.5, 0) {$\ldots v_i \ldots$};
			\node[draw, circle] (n5) at (6, 0) {$s$};
			
			\draw[-] (n1) -- (n2) node[midway, below] {$\frac{1}{2}$};
			\draw[-] (n2) -- (n3) node[midway, below] {$\frac{1}{4}$};
			\draw[-] (n3) -- (n4) node[midway, below] {$\frac{1}{8}$};
			\draw[-] (n4) -- (n5) node[midway, below] {$\frac{1}{2^i}$};
			
			\draw[bend left=60] (n1) to node[midway, above] {$1$} (n5);
			\draw[bend left=60] (n2) to node[midway, above] {$\frac{1}{2}$} (n5);
			\draw[bend left=60] (n3) to node[midway, below] {$\frac{1}{4}$} (n5);
			%\draw[bend left=315] (n4) to node[midway, below] {$\frac{1}{2^i}$} (n5);
		\end{tikzpicture}
		\caption{A canonical example using the sequence $(1-1/2^i)_i$ with its limit, $1.$ The vertices $v_i$ represent the element $\{1-1/2^i\}$ of the convergent sequence, and $s$ represents the boundary vertex, that is, $s = \lim_i v_i.$ The edge values are weights between the vertices.}
		\label{fig:grfini1}
	\end{figure}
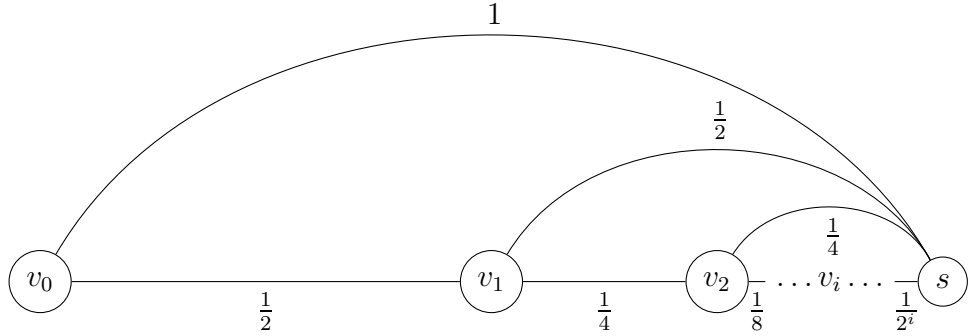
	
\end{Example}

\subsection{Weighted $p-$shortest  path metrics an integral $p-$averages  on graphs}

Following the idea that leads to the definition of general path metrics, we can use the integrals appearing in the characterization theorems explained in Section \ref{secfac} to introduce a new family of pseudometrics that are closely related to the metric symmetry on graphs. Broadly speaking, we will compute the shortest paths between two vertices in a graph considering the distance between two adjacent vertices given by an integral $p-$average as the ones appearing in previous section. We will focus the attention on the weighted $p-$shortest path metrics, which will be used in the next section, although some of the results presented here hold for general metrics on graphs.

Our main motivation is that in metric modeling, indices on graphs are often defined as real Lipschitz functions acting on the set of vertices of a given graph.
The interested reader can find many such applications, for example, in subjects as different as chemistry \cite{das} and network analysis in the social sciences \cite{bran}. 

Let us fix first some definitions.

\begin{Definition}
	Two points $x_1,x_2$ in a metric space $(M,D)$ are metrically symmetric with respect to a certain set $S$ if $d(x_1,y)=d(x_2,y)$ for every $y \in S.$ 
\end{Definition}

This property will be characterized in what follows in terms of domination by an integral $p-$average. Let us show first some other metric definitions and the relation between them and some already defined functions.

Let $G=(V,E,d),$ be a(n) (undirected connected) metric graph $G=(V,E,d).$   Recall that we assume that all the vertices of the graph can be connected by a finite path.  Let us define below a generalized version of the $p-$weighted shortest path metric for a graph starting from any distance $d.$ 

\begin{Definition}
	For $p \ge 1,$  we define the $p-$shortest path pseudodistance between two vertices $x_1,x_2 \in V$ associated to the metric $d$ by
	$$
	d_{p}(v_1,v_2) := 
	\inf \Big\{ \Big( \sum_{i=1}^{n-1}  d(v^i,v^{i+1})^p \, d \mu \Big)^{1/p}: \text{path}
	\, v_1=v^1\to v^2, \cdots, v^{n-1} \to v^n=v_2
	\Big\},
	$$
	where $v^{i} \to v^{i+1}$ means that $v^{i}$ and $v^{i+1}$ are connected. 
\end{Definition}

It can be easily seen that, in general, we cannot assure that $d_p \le d$ for any distance $d.$ Indeed, consider the (countable infinite) graph in which $V$ is identified with the set $\mathbb N$ of natural numbers, and the adjacent vertices are only the consecutive natural numbers. Consider the discrete distance in it ($d(v_1,v_2)=1$ whenever $v_1 \ne v_2.$) If we compute the pseudometric $d_p$ between $1 \in \mathbb N$ and $n \in \mathbb N,$ we have that $d_p(1,n)=(n-1)^{1/p},$ while $d(1,n)=1.$ Since $n$ is arbitrary, we get that this inequality cannot follow in general.

However, for the case when $d$ is a weighted shortest path metric $q_r,$ $r \ge 1,$ we always  get this inequality for the adequate values of $p.$ We assume that the weights $W = \{w(v_i,v_j): v_i \, \text{adjacent to} \, v_j \}$ between adjacent vertices of the graph make $q_r$ a metric, and not just a pseudometric.

\begin{Lemma} \label{lem1p}
	Fix $ 1 \le r \le p.$
	For every graph $G=(V,E),$ endowed with a weighted path metric $q_r,$ we have 
	$$
	(q_r)_{p}(v_1,v_2) \le q_r(v_1,v_2), \quad v_1, v_2 \in V.
	$$
\end{Lemma}
\begin{proof}
	Take $\varepsilon >0$ and a path $v_1=v^1 \to v^2 \to \cdots v^{n-1} \to v^n=v_2$ such that
	$$
	\Big( \sum_{i=1}^{n-1}  w(v^i,v^{i+1})^r \Big)^{1/r} \le q_r(v_1,v_2) + \varepsilon,
	$$
	and note that, since by definition for adjacent vertices $v^i$ and $v^{i+1}$ we have $q_r(v^i,v^{i+1}) \le w(v^i,v^{i+1}),$  we get by the inequality above
	$$
	(q_r)_{p}(v_1,v_2) := 
	\inf \Big\{ \Big( \sum_{i=1}^{n-1}  q_r(v^i,v^{i+1})^p \, d \mu \Big)^{1/p}: \text{path}
	\, v_1=v^1\to v^2, \cdots, v^{n-1} \to v^n=v_2
	\Big\}
	$$
	$$
	\le 
	\Big( \sum_{i=1}^{n-1}  w(v^i,v^{i+1})^p  \Big)^{1/p}.
	$$
	Thus, taking into account that $r \le p,$ we obtain
	$$
	(q_r)_{p}(v_1,v_2) \le \Big( \sum_{i=1}^{n-1}  w(v^i,v^{i+1})^p  \Big)^{1/p} \le \Big( \sum_{i=1}^{n-1}  w(v^i,v^{i+1})^r  \Big)^{1/r}  \le q_r(v_1,v_2) + \varepsilon.
	$$
	Since this happens for every $\varepsilon >0,$ we get the result.
	
\end{proof}

It can be easily seen that for $p=r$ we get $(q_p)_p=q_p,$ and so the inequality in Lemma \ref{lem1p} becomes an equality.

Let us define below the associated $p-$average notion of the weighted shortest path distance for a graph starting from any distance $d.$ Let $S \subset V$ be a compact subset of  vertices of a(n) (undirected connected) metric graph $G=(V,E,d).$ %and consider a probability measure $\mu \in \mathcal M(S).$ 

\begin{Definition}
	For $p \ge 1$ and  a probability measure $\mu \in \mathcal M(S),$  we define the $p-$shortest path pseudodistance between two vertices $v_1,v_2 \in V$ by
	$$
	d_{p,\mu}(v_1,v_2) := 
	$$
	$$
	\inf \Big\{ \Big( \sum_{i=1}^{n-1} \int_S \big| d(v^i,w)-d(v^{i+1},w) \big|^p \, d \mu \Big)^{1/p}: \text{path}
	\, v_1=v^1\to v^2, \cdots, v^{n-1} \to v^n=v_2
	\Big\}
	$$
	$$
	= \inf_{pahts} \left\| \Big( \sum_{i=1}^{n-1} \big| d(v^i,\cdot)-d(v^{i+1},\cdot) \big|^p  \Big)^{1/p} \right\|_{L^p(\mu)},
	$$
	where $v^{i} \to v^{i+1}$ means that $v^{i}$ and $v^{i+1}$ are connected.
\end{Definition}

\begin{Lemma} \label{lem2p}
	For every metric graph $G=(V,E,d),$ every $p \ge 1$  and every  probability measure 
	$\mu \in \mathcal M(S),$ 
	$$
	d_{p,\mu}(v_1,v_2) \le d_p(v_1,v_2), \quad v_1, v_2 \in V.
	$$
\end{Lemma}
\begin{proof}
	This is a consequence of a direct calculation. Indeed, note that by the triangular inequality for $d,$ and taking into 
	account that  $\mu$ is a probability 
	measure, we obtain
	$$
	\inf \Big\{ \Big( \sum_{i=1}^{n-1} \int_S \big| d(v^i,w)-d(v^{i+1},w) \big|^p \, d \mu \Big)^{1/p}: \text{path}
	\, v_1=v^1\to v^2, \cdots, v^{n-1} \to v^n=v_2
	\Big\}
	$$
	$$
	\le \inf \Big\{ \Big( \sum_{i=1}^{n-1}  d(v^i,v^{i+1})^p \, d \mu \Big)^{1/p}: \text{path}
	\, v_1=v^1\to v^2, \cdots, v^{n-1} \to v^n=v_2
	\Big\} = d_p(v_1,v_2).
	$$
\end{proof}

A straightforward consequence of Lemmas \ref{lem1p} and \ref{lem2p} is the following

\begin{Proposition}
	%$p-$concave
	Let $S \subset V$ be a compact subset of a weighted undirected connected graph $G=(V,E).$ Fix $1 \le r \le p.$ Then for  every  probability measure 
	$\mu \in \mathcal M(S),$ 
	$$
	(q_r)_{p,\mu}(v_1,v_2) \le (q_r)_p(v_1,v_2) \le q_r(v_1,v_2), \quad v_1, v_2 \in V.
	$$
\end{Proposition}

\subsection{Eccentric $p-$summing domination for functionals on graphs and metric symmetry}
In this section we show how domination of particular real Lipschitz functions can be used for the characterization of indices of metric symmetry on graphs. %Indeed, take a countable graph and a Borel regular probability measure.
Let us motivate this application with two examples.

\begin{Example}
	Consider a graph defined by the following points of $\mathbb R^2$ as vertices. Take $v_1=(-1,0),$ $v_2=(1,0)$ and $v_k = a_{k-2}$ for $k \ge 2,$  where $(a_k)_{k=1}^\infty$ is a convergent sequence formed by elements as $(0,s_k)$ together with is limit $(0,s_\infty) \in \mathbb R^2.$ Suppose that the connectivity is given as in Example \ref{exfin1} for the elements of the sequence and its limit, and the elements $v_1$ and $v_2$ are connected with all the elements of the sequence, and with its limit too. The weights for the metric are again the size of the segments connecting every couple of adjacent vertices. It is clearly a compact graph with the weighted shortest path metric.
	\begin{figure}[ht]
		\centering
		\begin{tikzpicture}
			\node[draw, circle] (n1) at (-6, 0) {$v_3$};
			\node[draw, circle] (n2) at (0, 0) {$v_4$};
			\node[draw, circle] (n3) at (3, 0) {$v_5$};
			\node[draw=none] (n4) at (4.5, 0) {$\ldots v_k \ldots$};
			\node[draw, circle] (n5) at (6, 0) {$s$};
			\node[draw, circle] (n6) at (-6, -4) {$v_1$};
			\node[draw, circle] (n7) at (-6, 4) {$v_2$};
			
			\draw[-] (n1) -- (n2) node[midway, below] {};
			\draw[-] (n2) -- (n3) node[midway, below] {};
			\draw[-] (n3) -- (n4) node[midway, below] {};
			\draw[-] (n4) -- (n5) node[midway, below] {};
			\draw[-] (n6) -- (n1) node[midway, below] {};
			\draw[-] (n6) -- (n2) node[midway, below] {};
			\draw[-] (n6) -- (n3) node[midway, below] {};
			%\draw[-] (n6) -- (n4) node[midway, below] {};
			\draw[-] (n6) -- (n5) node[midway, below] {};
			\draw[-] (n7) -- (n1) node[midway, below] {};
			\draw[-] (n7) -- (n2) node[midway, below] {};
			\draw[-] (n7) -- (n3) node[midway, below] {};
			%\draw[-] (n7) -- (n4) node[midway, below] {};
			\draw[-] (n7) -- (n5) node[midway, below] {};

			\draw[bend left=60] (n1) to node[midway, below] {} (n5);
			\draw[bend left=60] (n2) to node[midway, above] {} (n5);
			\draw[bend left=60] (n3) to node[midway, below] {} (n5);
		\end{tikzpicture}
		\caption{As in Figure \ref{fig:grfini1} of Example \ref{exfin1}, in this graph the convergent sequence $(a_k)_{k=1}^\infty$ is formed by the elements $(0, 1-1/2^k)_k$ together with its limit $s = (0, 1).$ In this figure the weights of the graph are not represented.}
		\label{fig:grfini2}
	\end{figure}
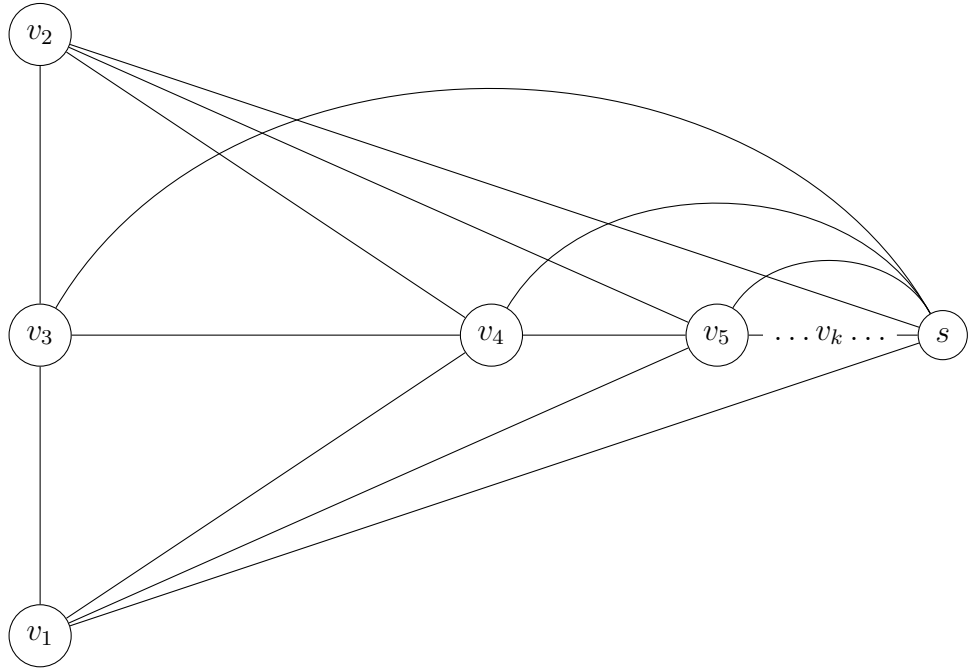
	
	Take the Borel regular measure $\mu$ on the whole set $V$ of vertices of the graph that is written as a series 
	$$
	\mu = \sum_{i=3}^\infty \alpha_i \delta_{\{v_i\}},
	$$
	where $\delta_{v_i}$ is the Dirac measure on the vertex $v_i$ and $\alpha_i$ are non-negative real numbers such that $\sum_{i=3}^\infty \alpha_i=1.$
	
	Suppose that a real valued Lipschitz function $f:V \to \mathbb R$ satisfies a domination as
	$$
	\big| f(v_i)- f(v_j) \big| \le K \int \big| d(v_i,v)- d(v_j,v) \big| \, d \mu(v).
	$$
	Note that this implies that $f(v_1)=f(v_2),$ since by symmetry $d(v_1,v)= d(v_2,v)$ for all $v=v_i$ for $i=3,4,...$, and for $v=v_1$ or $v=v_2,$  $\big| d(v_i,v)- d(v_j,v) \big|=1,$ but $\mu(\{v\})=0$ for these points. This means that all Lipschitz functionals dominated by such an integral are symmetric, in the sense that $f(v_1)=f(v_2).$
\end{Example}

\begin{Example}
	A noncountable graph. 
	Consider all the points of the circle of radius $1$ together with its center, that is supposed to be the point $v_0=(0,0),$ as a subset of $\mathbb R^2.$ Suppose that all the points of the circle are connected, and also with the center $v_0.$ Consider the sets of weights $w(v_i,v_j)$ given by the size of the arc connecting two points of the circle, and $w(v_0,v_i)=1$ for all the points $v_i$ in the circle. This is a metric graph with the shortest path metric for which the distance $d(v_i,v_j)$ equals again the size of the arc connecting them if they belong to the circle, and $1$ if one of the points involved is $v_0.$
	\begin{figure}[ht]
		\centering
		\begin{tikzpicture}
			%\draw[thick] (0,0) circle(3);
			\draw (-2.9,-0.3) arc[start angle=184, end angle=270, radius=2.9];
			\draw (0,-3) arc[start angle=-90, end angle=0, radius=3];
			\draw (3,0) arc[start angle=0, end angle=83, radius=3];
			
			\node[draw, circle] (v0) at (0,0) {$v_0$};
			\node[draw, circle] (v1) at (-3,0) {$v_i$};
			\node[draw, circle] (v2) at (-2.12,2.12) {$v_j$};
			\node[draw, circle] (v3) at (0,3) {$v_k$};
			
			\node[draw=none] (n4) at (-1.5,-1.5) {$\uldots$};
			\node[draw=none] (n5) at (1.5,-1.5) {$\uddots$};
			\node[draw=none] (n6) at (1.5,1.5) {$\ddots$};
			\node[draw=none] (n7) at (0,-2) {$\cdots$};
			\node[draw=none] (n8) at (2,0) {$\vdots$};
			
			\draw (v0) -- (v1) node[midway, above] {$1$};
			\draw (v0) -- (v2) node[midway, above] {$1$};
			\draw (v0) -- (v3) node[midway, right] {$1$};
			\draw (v0) -- (n4) node[midway, below] {$1$};
			\draw (v0) -- (n5) node[midway, below] {$1$};
			\draw (v0) -- (n6) node[midway, below] {$1$};
			
			\draw (v1) -- (v3) node[midway, above] {$\frac{\pi}{2}$};
			\draw (v1) -- (n4) node[midway, above] {};
			\draw (v3) -- (n6) node[midway, above] {};
			
			%\draw (v2) -- (n4) node[midway, above] {};
			%\draw (v2) -- (n6) node[midway, above] {};
			
			\draw (-2.9,0.3) arc[start angle=176,end angle=144,radius=2.9] node[midway, left] {$\frac{\pi}{4}$};
			\draw (-1.83,2.33) arc[start angle=130,end angle=98,radius=2.9] node[midway, above] {$\frac{\pi}{4}$};

		\end{tikzpicture}
		\caption{Graph on the circle with center $v_0$ and arc weights.}
		\label{fig:circlegraph}
	\end{figure}
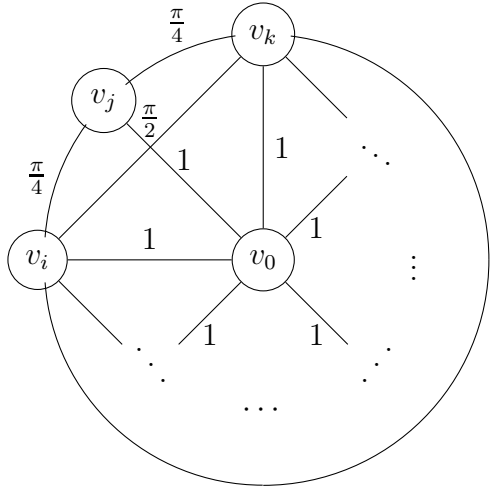

	Take the measure $\mu= \delta_{\{v_1\}}.$ Then  every Lipschitz functional dominated as
	$$
	\big| f(v_i)- f(v_j) \big| \le K \int \big| d(v_i,v)- d(v_j,v) \big| \, d \delta_{\{v_0\}} (v)
	$$
	for any $K>0$ satisfies that $f(v_i)=f(v_j)=Q$ for all $v_i, v_j$ in the unit circle (that is, it takes a constant value $Q$), and $|f(v_0)-f(v_i)| \le K,$  that is, $Q-K \le f(v_0) \le Q +K.$
\end{Example}

These examples show that the integral domination gives information about the properties of the functionals on certain paths connecting concrete vertices in a graph. The problem we face in this section is to give conditions under which a Lipschitz functional satisfies that  given two vertices there is a path connecting them that satisfy that the functional is almost constant, or at least its $p-$th variation is controlled.

Let $G=(V,E)$ be a weighted graphs with weights $W.$
Given a functional $f:V \to \mathbb R,$ we define its $p-$best path estimate as
$$
E_p(f)(v_1,v_2):=
$$
$$
\inf \Big\{ \Big( \sum_{i=1}^{n-1}  \big| f(v^i)-f(v^{i+1}) \big|^p \, d \mu \Big)^{1/p}: \text{path}
\, v_1=v^1\to v^2, \cdots, v^{n-1} \to v^n=v_2
\Big\}
$$
for every $v_1, v_2 \in V.$

\begin{Definition}
	Let $G=(V,E)$ be a weighted graphs with weights $W.$
	We say that a functional $f:V \to \mathbb R$ satisfies a $p-$th power integral estimate if there is a constant $Q>0,$ 
	a compact subset $S \subset V$ and a probability measure $\mu \in \mathcal M(S)$ such that for every $v_1, v_2 \in V,$
	$$
	E_p(f)(v_1,v_2) \le Q \, d_{p,\mu}(v_1,v_2).
	$$
\end{Definition}

Using the tools developed in other sections, next result shows how are the functionals that satisfy such definition.

\begin{Theorem} \label{t2}
	Let $f:V \to \mathbb R$ be a  
	Let $G=(V,E)$ be a weighted graphs with weights $W.$
	Let $f:V \to \mathbb R$ be a Lipschitz map. Then
	\begin{itemize}
		
		\item[1] $ E_p(f)(v_1,v_2) \le Lip(f)  \, d_p(v_1,v_2)$ for all $v_1, v_2 \in V.$ 
		
		\item[2] If $f$ is
		eccentric $p-$summing functional, then there is a compact subset $S$ and a probability measure $\mu \in \mathcal M(S)$ such that $f$ satisfies a $p-$th power integral estimate
		$$
		E_p(f)(v_1,v_2) \le C(f)  \, d_{p,\mu}(v_1,v_2).
		$$
		That is,  for every $v_1, v_2 \in V,$
		$$
		\inf \Big\{ \Big( \sum_{i=1}^{n-1}  \big| f(v^i)-f(v^{i+1}) \big|^p \, d \mu \Big)^{1/p}: \text{path}
		\, v_1=v^1\to v^2, \cdots, v^{n-1} \to v^n=v_2   \Big\}
		$$
		$$
		\le
		C(f) \,
		\inf \Big\{ \Big( \sum_{i=1}^{n-1} \int_S \big| d(v^i,w)-d(v^{i+1},w) \big|^p \, d \mu(w) \Big)^{1/p}: \text{path}
		\, v_1 \to v^2, \cdots,  \to v_2  \Big\},
		$$
		where $C(f)$ is the eccentric $p-$summing constant of $f.$
	\end{itemize}
\end{Theorem}
\begin{proof}
	1. Consider a pair $v_1, v_2 \in V$ and take a path connecting them $v_1=v^1\to v^2, \cdots, v^{n-1} \to v^n=v_2.$
	Then
	$$
	\Big( \sum_{i=1}^{n-1}  \big| f(v^i)-f(v^{i+1}) \big|^p \, d \mu \Big)^{1/p} \le 
	\Big( \sum_{i=1}^{n-1}  Lip(f) \, d(v^i,v^{i+1})^p \, d \mu \Big)^{1/p},
	$$
	and so $E_p(f)(v_1,v_2) \le Lip(f) \, d \Big( \sum_{i=1}^{n-1}  (v^i,v^{i+1})^p \, d \mu \Big)^{1/p}.$ Since this happens for every path connecting $v_1$ and $v_2,$ we get the result.
	
	2. Fix $v_1, v_2 \in V.$ Consider  a path connecting them $v_1=v^1\to v^2, \cdots, v^{n-1} \to v^n=v_2.$ Since $f$ is eccentric $p-$summing, by Theorem \ref{pietfunc}  we find a set $S,$ a constant $C>0$ and a measure
	$\mu \in \mathcal M(S)$ such that 
	$$
	|f( v)- f(s) | \leq C(f) \Big( \int_S | d(v,w) - d(s,w) |^p \,  d\mu(w) \Big)^{1/p}, \quad v,s \in V.
	$$
	Writing the $p-$sum in both sides, we obtain
	$$
	\Big( \sum_{i=1}^{n-1}  \big| f(v^i)-f(v^{i+1}) \big|^p \, d \mu \Big)^{1/p} \le 
	\Big( C(f)^p \, \sum_{i=1}^{n-1}   \int_S | d(x,w) - d(y,w) |^p \,  d\mu(w) \Big)^{1/p}, 
	$$
	and so
	$$
	E_p(f)(v_1,v_2) \le C(f) \Big( \sum_{i=1}^{n-1}   \int_S | d(v^i,w) - d(v^{i-1},w) |^p \,  d\mu(w) \Big)^{1/p}.
	$$
	Now we only need to consider the infimum in the right hand side to get the desired inequality.
\end{proof}

%p additive domination of functionals

Let us explicitly explain the application of the above results to graph analysis. Following the standard tool definition for metric graph analysis in applied science, we say that a function $I:V \to \mathbb R$ that quantitatively represents a certain property of the graph vertices is an index \cite{das}. It is usually assumed to be a Lipschitz function for a certain distance defined in the graph. In the metric modeling of graphs  (and in the modeling of metric spaces in general), it is common to use this type of functions to represent relevant properties of the graphs: for example, if one considers the graph of the cities of a region with the road connections between them, an index of interest could be given by the real function that maps the number of inhabitants of each city. Thus, such an index is defined as a Lipschitz functional on $V.$

% If we fix a compact subset $S \subseteq V$ and a measure $\mu$ on it, and we define the corresponding functional $d_{p,\mu},$ we

For a given measure $\mu,$ let us consider the pseudometric $d_{p,\mu}.$  Two vertices $v_1$ and $v_2$ in $V$
are symmetric  with respect to $\mu$ if $d_{p,\mu}(v_1,v_2)=0.$

The meaning of this relation is clear: two vertices $v_1$ and $v_2$ are symmetric with respect to $\mu$ if the associate metric functions $w \mapsto d(v_1,w)$ and $w \mapsto d(v_2,w)$ are $\mu-$almost everywhere equal, that is $\int_S |d(v_1,w)- d(v_2,w)|^p d \mu=0.$ The term symmetric is explained by the fact that, if this holds, then the vertices $v_1$ and $v_2$ have similar distribution of distances except in a set of vertices $w$ that is $\mu-$null. %This is represented by the set $Ker d_{p,\mu}.$

% If a Lipschitz index $I:V \to \mathbb R$ has $\mu$ as control measure and  we can consider the set 

Recall that, in the context of metric modelling, an index $I$ is a real Lipschitz function $I: V \to \mathbb R.$
Next result gives the formal property that characterizes the existence of $I-$constant paths for a certain index $I.$ The proof is a straightforward application of Theorem \ref{t2}, which provides the existence of a measure $\mu$ and a domination of $E_p(f)$ by $d_{p\,mu}(f).$

\begin{Corollary}
	Let $I:V \to \mathbb R$ be an index on a weighted undirected connected graph $G=(V,E).$ If $I$ is eccentric $p-$summing, then there is a compact set $S$ and a probability measure $\mu \in \mathcal M(S)$ such that if two vertices in $V$ are symmetric with respect to $\mu$, then there is a path connecting $v_1$ and $v_2$ such that $I(v_1)=I(v)=I(v_2)$ for all vertices $v$ in the path.
\end{Corollary}

The approximate version of this result also holds, in the sense that if the vertices $v_1$ and $v_2$ are not symmetric with respect to $\mu,$ but satisfy that $d_{p,\mu}(v_1,v_2)$ is small, we obtain that  the $p-$variation $E_p(I)$ is small because $E_p(I) \le C(f) \, d_{p\,mu}(I).$

To finish, let us return to the example of the graph of the cities of a region with the road connections between them, and give the weights of the edges by ordering the distances between them (e.g. 1 if the distance is 0 to 5 km, 2 if the distance is 6 to 10, and so on). Let us fix an index $I$ that orders the set of vertices (cities) by thousands of inhabitants: $I(v)=1$ if the city $v$ has from $1$ to $1000$ inhabitants, $I(v)=2$ if $v$ has from $1001$ to $2000$ inhabitants, and so on.
Suppose $I$ is controlled by an appropriate measure $\mu$ modeled in order to control the routes between two cities to reduce risk by restricting crossing large cities; for example, the support of the measure $\mu$ contains only the largest cities, so distances to small cities is irrelevant. Given that there exists a constant $R$ such that $E_p(I) \le R \, d_{p,\mu},$ then a small value of $d_{p,\mu}(v_1,v_2)$ for two small cities (populations less than $1000$ inhabitants, $I(v_1)=I(v_2)=1),$ means that we can find a path connecting $v_1$ and $v_2$ such that $E_p(I)(v_1,v_2)$ is also small (or $0$). For instance, if it is $0,$ it means that there is a path connecting cities $v_1$ and $v_2$ that crosses only towns of less than $1000$ inhabitants, since for all vertices that the path crosses, $I(v)=I(v_1)=1.$

\vspace{0.5cm}

%
%Recall that, if $q$ is a pseudometric $q$ on $V,$ we write as usual $Ker q$ for the set of all the points $y \in M$ which satisfy  $q(x,y)=0,$ that is
% $$
% Ker q= \Big\{ y \in V: q(x,y)=0 \Big\}.
% $$

%%%%%%%paper anterior del tema con roger

\section{Conclusions}

In this paper, we introduce a new type of 
$p-$summability for Lipschitz maps between metric spaces, which we call to be eccentric $p-$summing. Through its detailed characterization, we explore the general summation properties of metric spaces, introducing some related integral inequalities for real Lipschitz functionals. These inequalities lead to factorization theorems for general Lipschitz operators. 
We establish some key definitions related to summing metrics for sequences of metric spaces, laying the groundwork for our new notion of operator summability.
Using them, we focus on proving the main results concerning summing Lipschitz operators, providing not only the associated factorization theorems but also discussing general cases with concrete examples to illustrate the breadth of applicability. These results extend classical summability concepts, offering new insights into the behavior of Lipschitz operators across diverse settings.

The final section of the paper addresses the application of these summing Lipschitz operators in the context of metric graphs. Given their potential applications in various fields, we show how the integral inequalities that characterize summing properties can be adapted to compact metric graphs. This extension provides a powerful new tool for investigating the concept of metric symmetry in graphs. Furthermore, we present examples and discuss new results regarding the characterization of approximate metric symmetry, utilizing analytic techniques in contrast to the classical algebraic approaches. By shifting the focus to analytic methods, we offer a new perspective on the study of symmetry in metric spaces, with implications for both theoretical research and practical applications in graph theory.

\vspace{1cm}

\section*{Acknowledgment}
	The first author was supported by a contract of the Programa de Ayudas de Investigaci\'on y Desarrollo (PAID-01-21), Universitat Polit\`ecnica de Val\`encia.
	This research was funded by the Agencia Estatal de Investigaci\'on, grant number PID2022-138342NB-I00.
	The research was funded by the European Union’s Horizon Europe research and innovation program under the Grant Agreement No. 101059609 (Re-Livestock).

\bibliographystyle{plainnat}
\bibliography{biblio_psumming}

\begin{thebibliography}{10}

\bibitem{ach2}
D.~Achour, E.~Dahia, and P.~Turco.
\newblock The lipschitz injective hull of lipschitz operator ideals and
  applications.
\newblock {\em Banach J. Math. Anal.}, 14:1241--1257, 2020.

\bibitem{ach1}
D.~Achour, P.~Rueda, and R.~Yahi.
\newblock $(p, \sigma)-$absolutely lipschitz operators.
\newblock {\em Ann. Funct. Anal.}, 8(1):38--50, 2017.

\bibitem{ang}
J.~C. Angulo-López and M.~Fernández-Unzueta.
\newblock Lipschitz p-summing multilinear operators.
\newblock {\em J. Funct. Anal.}, 279:108572, 2020.

\bibitem{are}
R.~F. Arens and J.~Jr. Eells.
\newblock On embedding uniform and topological spaces.
\newblock {\em Pacific J. Math.}, 6:397–403, 1956.

\bibitem{arn}
R.~Arnau, J.~M. Calabuig, and E.~A. Sánchez~Pérez.
\newblock Representation of lipschitz maps and metric coordinate systems.
\newblock {\em Mathematics}, 10(20):3867, 2022.

\bibitem{bran}
U.~Brandes.
\newblock {\em Network analysis: methodological foundations}.
\newblock Springer Science \& Business Media, Berlin, 2005.

\bibitem{buck}
F.~Buckley and F.~Harary.
\newblock {\em Distance in graphs}.
\newblock Addison-Wesley, Boston, 1990.

\bibitem{chen}
D.~Chen and B.~Zheng.
\newblock Remarks on lipschitz $p-$summing operators.
\newblock {\em Proc. Amer. Math. Soc.}, 139(8):2891--2898, 2011.

\bibitem{cha}
J.~A. Chávez-Domínguez.
\newblock Duality for lipschitz p-summing operators.
\newblock {\em J. Funct. Anal.}, 261(2):387--407, 2011.

\bibitem{cha2}
J.~A. Chávez-Domínguez.
\newblock Lipschitz $(q,p)-$mixing operators.
\newblock {\em Proc. Amer. Math. Soc.}, 140(9):3101--3115, 2012.

\bibitem{cob}
{\c{S}}.~Cobza{\c{s}}, R.~Miculescu, and A.~Nicolae.
\newblock {\em Lipschitz Functions}.
\newblock Springer International Publishing, Berlin/Heidelberg, Germany, 2019.

\bibitem{das}
K.~C. Das, I.~Gutman, and B.~Furtula.
\newblock Survey on geometric-arithmetic indices of graphs.
\newblock {\em MATCH Commun. Math. Comput. Chem.}, 65:595--644, 2011.

\bibitem{deflo}
A.~Defant and K.~Floret.
\newblock {\em Tensor norms and operator ideals}.
\newblock Elsevier, Amsterdam, 1992.

\bibitem{deza}
M.~M. Deza and E.~Deza.
\newblock {\em Encyclopedia of distances}.
\newblock Springer, Berlin Heidelberg, 2009.

\bibitem{djt}
J.~Diestel, H.~Jarchow, and A.~Tonge.
\newblock {\em Absolutely Summing Operators}.
\newblock Cambridge University Press, Cambridge, 1995.

\bibitem{farmer}
J.~Farmer and W.~Johnson.
\newblock Lipschitz $p-$summing operators.
\newblock {\em Proc. Amer. Math. Soc.}, 137:2989--2995, 2009.

\bibitem{fer}
M.~Fernández-Unzueta.
\newblock Lipschitz $p-$summing multilinear operators correspond to lipschitz
  $p-$summing operators.
\newblock {\em Proc. Amer. Math. Soc.}, 151(1):215--223, 2023.

\bibitem{god}
G.~Godefroy.
\newblock A survey on lipschitz-free banach spaces.
\newblock {\em Comment. Math.}, 55:89--118, 2015.

\bibitem{gods}
C.~Godsil and G.~F. Royle.
\newblock {\em Algebraic graph theory}.
\newblock Springer Science \& Business Media, Berlin, 2001.

\bibitem{pan}
V.~M. Panaretos and Y.~Zemel.
\newblock {\em An invitation to statistics in Wasserstein space}.
\newblock Springer Nature, Berlin, 2020.

\bibitem{piets}
A.~Pietsch.
\newblock {\em Operator Ideals}.
\newblock North-Holland, Amsterdam, 1980.

\bibitem{rodsan}
J.~Rodríguez-López and E.~A. Sánchez-Pérez.
\newblock Power-aggregation of pseudometrics and the mcshane-whitney extension
  theorem for lipschitz p-concave maps.
\newblock {\em Acta Appl. Math.}, 1709:611--629, 2020.

\bibitem{wilson1935}
W.~Wilson.
\newblock On certain types of continuous transformations of metric spaces.
\newblock {\em Am. J. Math.}, 57:62--68, 1935.

\end{thebibliography}

\end{document}